\documentclass[11pt]{amsart}
\usepackage{amscd}
\usepackage{verbatim}
\usepackage{amssymb}
\usepackage{amsmath}
\usepackage{amsthm}
\usepackage{amsfonts}
\usepackage{mathrsfs}
\usepackage{amsthm}
\usepackage{euscript}
\usepackage{xcolor}

\usepackage[margin=30 mm,heightrounded=true,centering]{geometry}
\newtheorem{theorem}{Theorem}[section]
\newtheorem{proposition}[theorem]{Proposition}
\newtheorem{lemma}[theorem]{Lemma}
\newtheorem{remark}[theorem]{Remark}
\newtheorem{corollary}[theorem]{Corollary}
\newtheorem{definition}[theorem]{Definition}

\DeclareMathOperator{\riem}{Riem}
\DeclareMathOperator{\ric}{Ric}
\DeclareMathOperator{\hess}{Hess}
\DeclareMathOperator{\tcd}{TCD}
\DeclareMathOperator{\ncd}{NCD}
\DeclareMathOperator{\id}{id}
\DeclareMathOperator{\tr}{tr}

\DeclareMathOperator{\divergence}{div}
\DeclareMathOperator{\vol}{vol}
\DeclareMathOperator{\edge}{edge}

\newcounter{mnotecount}[section]

\begin{document}

\title[Curvature-dimension bounds for Lorentzian splitting theorems]{Curvature-dimension bounds for Lorentzian splitting theorems}

\author{Eric Woolgar}
\address{Department of Mathematical and Statistical Sciences,
University of Alberta, Edmonton, Alberta, Canada T6G 2G1}
\email{ewoolgar(at)ualberta.ca}

\author{William Wylie}
\address{215 Carnegie Building, Department of Mathematics, Syracuse University, Syracuse, NY 13244, USA}
\email{wwylie@syr.edu}

\date{\today}

\begin{abstract}
We analyze Lorentzian spacetimes subject to curvature-dimension bounds using the Bakry-\'Emery-Ricci tensor. We extend the Hawking-Penrose type singularity theorem  and the Lorentzian timelike splitting theorem  to synthetic dimensions $N\le 1$, including all negative synthetic dimensions. The rigidity of the timelike splitting reduces to a warped product splitting when $N=1$. We also extend the null splitting theorem of Lorentzian geometry, showing that it holds under a null curvature-dimension bound on the Bakry-\'Emery-Ricci tensor for all $N\in (-\infty, 2]\cup (n,\infty)$ and for the $N=\infty$ case as well, with reduced rigidity if $N=2$. In consequence, the basic singularity and splitting theorems of Lorentzian Bakry-\'Emery theory now cover all synthetic dimensions for which such theorems are possible. The splitting theorems are found always to exhibit reduced rigidity at the critical synthetic dimension.
\end{abstract}

\maketitle

\setcounter{equation}{0}

\section{Introduction}
\setcounter{equation}{0}

\noindent Ricci comparison theory is one of the most important tools of Riemannian geometry. In the Lorentzian setting, the analogous tools and techniques lead to singularity theorems, which have had a profound impact in general relativity. The discovery of the singularity theorems suggested to physicists that the theory of gravitation based on Einstein's general relativity required modification. The obvious modification, quantization, has proved to be inordinately difficult (though deeply interesting), and for this reason and others, many modified classical gravitation theories have been proposed.

Some of these proposals have natural analogues in the Riemannian setting. Consider perhaps the best known example, the inclusion of a scalar field non-minimally coupled to Einstein's relativity theory. The primary example Brans-Dicke theory,\footnote
{We also note the strongly related example of `dilatons' in warped product Kaluza-Klein models. The most basic examples can be regarded as special cases of the Brans-Dicke theory, though typical physics models usually contain other fields as well.}
which also predicts singularities if one casts the theory in a suitable conformal gauge so that it becomes the Einstein theory with non-universal matter couplings, and applies energy conditions in that conformal gauge. However, this does not imply that singularities are unavoidable if energy conditions are imposed in other conformal gauges which may be viewed as more natural. It is therefore interesting to ask whether the singularity theorems hold only in the Einstein theory and perhaps in other classical theories that can be made to resemble it, or whether they are more general.

The geometric setting for non-minimal scalar-tensor gravitation is but one of the uses of Bakry-\'Emery geometry (among the others, the static Einstein equations have a Bakry-\'Emery description, and the so-called \emph{near horizon geometries} satisfy a very similar but somewhat more generalized theory\footnote
{We are grateful to Marcus Khuri for bringing this to our attention.}
). We recall that the $N$-Bakry-\'Emery-Ricci tensor, or simply the $N$-Bakry-\'Emery tensor, is a generalization of the Ricci tensor $\ric$ of a metric $g$ on an $n$-manifold $M$. If in addition to $g$ we are given a real number $N\neq n$ called the \emph{synthetic dimension} and a twice-differentiable function $f:M\to{\mathbb R}$, the $N$-Bakry-\'Emery-Ricci tensor is
\begin{equation}
\label{eq1.1}
\ric^N_f:=\ric+\hess f-\frac{df\otimes df}{N-n}\ .
\end{equation}
Here $\hess$ denotes the Hessian defined by the Levi-Civita connection $\nabla$ of the metric $g$ by $\hess u :=\nabla^2 u$. The synthetic dimension derives its name from the fact that, when $N>n$ is an integer, \eqref{eq1.1} is the expression for the Ricci curvature of an $N$-dimensional warped product over $(M,g)$, but $N$ need not be an integer, nor need it be greater than $n$; indeed, it need not be positive. There is also a tensor called simply the Bakry-\'Emery-Ricci (or more simply Bakry-\'Emery) tensor, given by
\begin{equation}
\label{eq1.2}
\ric_f:=\ric+\hess f\ .
\end{equation}

There is by now a well-developed version of Bakry-\'Emery-Ricci comparison theory \cite{Lott, GWW, Wylie}, in which bounds on the Ricci tensor are phrased as \emph{curvature-dimension inequalities}. This leads us to ask whether the singularity theorems hold when conditions of curvature-dimension type replace and generalize so-called energy conditions in the Lorentzian setting.

Such questions were first asked in \cite{Case}, where a Hawking-Penrose type singularity theorem is proved, as is an $N$-Bakry-\'Emery version of the Lorentzian Cheeger-Gromoll splitting theorem when $N>n$ or when $f\le k$ and $N=\infty$. Singularity theorems of cosmological type were found in \cite{GW} and \cite{EWW}, as were splitting theorems in the rigidity cases. These cases arise when conditions are arranged so that singularities are avoided. Splitting theorems show that this can only occur when the geometry is of a special split type, typically a Lorentzian product, warped product, or perhaps a twisted product. An interesting feature of the Bakry-\'Emery theory is that the rigidity is somewhat relaxed for a critical value of the synthetic dimension $N$.

At this point, a fairly complete picture is developing. There remain, however, several open issues. Most obvious among them is the extension of some of Case's results to negative (and small positive) synthetic dimension $N$, as well as the generalization to null rather than timelike curvature-dimension conditions. In this paper, we are able to extend the timelike splitting theorem of \cite{Case} to $N<1$, including negative $N$, and to $N=1$ with the optimal weaker warped product rigidity. We are also able to prove a Bakry-\'Emery version of Galloway's null splitting theorem \cite{Galloway}. The latter theorem has the feature that the critical synthetic dimension in which rigidity relaxes is $N=2$, whereas it is $N=1$ for the timelike theorem.

We recall the following definitions.
\begin{definition}\label{definition1.1}
Given functions $f$ and $\lambda$ and a real number $N$ (the synthetic dimension), if $\ric^N_f(X,X)\ge \lambda$ for all unit timelike vectors $X$ (i.e., $g(X,X)=-1$) then we say that the \emph{timelike curvature-dimension condition} $\tcd(\lambda,N)$ holds for $(M,g,f)$.
\end{definition}
We note that $\lambda$ is usually taken to be constant. We also state an analogous definition in terms of null vectors, but only for $\lambda=0$ to make the definition rescaling invariant.
\begin{definition}\label{definition1.2}
If $\ric^N_f(X,X)\ge 0$ for all null vectors $X$ (i.e., $g(X,X)=0$) and a given function $f$, we say that the \emph{null curvature-dimension condition} $\ncd(N)$ holds for $(M,g,f)$.
\end{definition}

For the critical synthetic dimensions $N=1$ in the timelike case and $N=2$ in the null case,  there are natural associated projective and conformal structures respectively. (See Section \ref{Section4}  for further details.) We define the notion of $f$-completeness to be the respective completeness conditions for these structures.  These both turn out to be integral conditions for the potential function $f$ along inextendible geodesics.

\begin{definition}\label{definition1.3}
We say a future-inextendible timelike geodesic $\gamma:[0,T)\to M$, $T\in (0,\infty]$, is \emph{future $f$-complete} if it is complete with respect to the parameter $s(t):=\int\limits_0^t e^{-\frac{2f(\tau)}{(n-1)}}d\tau$, $t\in [0,T)$, where we abbreviate $f(\tau):=f\circ\gamma(\tau)$.  We say a future-inextendible null geodesic $\gamma:[0,T)\to M$, $T\in (0,\infty]$, is \emph{future $f$-complete} if it is complete with respect to the parameter ${\tilde s}(t):=\int\limits_0^t e^{-\frac{2f(\tau)}{(n-2)}}d\tau$, $t\in [0,T)$.  \emph{Past $f$-completeness} is defined dually. A timelike or null geodesic that is both future and past $f$-complete is said to be simply \emph{$f$-complete}. A spacetime obeys the \emph{timelike (null, nonspacelike) $f$-completeness condition} if every complete timelike (null, nonspacelike) geodesic is $f$-complete.\footnote
{Note that a spacetime may obey an $f$-completeness assumption while containing a geodesic $\gamma$ which is not $f$-complete, provided $\gamma$ is also not complete. Also, $f$-completeness of a geodesic does not imply that the geodesic is complete if $f$ is unbounded.}
\end{definition}

We are now in a position to state our main theorems. The first theorem extends Case's $f$-Bakry-\'Emery Hawking-Penrose singularity theorem \cite[Theorem 4.6]{Case} to negative synthetic dimension, while also weakening his boundedness condition on $f$ in the $N=\infty$ (in our nomenclature, $N=-\infty$)\footnote
{Since $N=\infty$ and $N=-\infty$ denote the same limit, we generally denote this limit by $N=-\infty$ except when referring to \cite{Case}.}
case.

\begin{theorem}\label{theorem1.4}
Let $(M,g)$ be a chronological spacetime with $\dim M =: n\ge 3$ satisfying the $f$-generic curvature condition (see Definition \ref{definition2.1}), the $\tcd(0,N)$ condition with $N\in [-\infty,1]\cup (n,\infty)$ and, for $N\in [-\infty,1]$, the $f$-completeness condition. Then $(M,g)$ is nonspacelike geodesically incomplete if any one of the following conditions holds:
\begin{enumerate}
\item $M$ has a closed $f$-trapped surface.
\item $(M,g)$ has a point $p$ such that every inextendible null geodesic through $p$ is $f$-reconverging somewhere.
\item $M$ has a compact spacelike hypersurface.
\end{enumerate}
\end{theorem}

We note several points. First, we remind the reader that $\tcd(0,N)$ implies $\ncd(N)$ by continuity. Next, note that the $f$-complete condition holds whenever $f$ has an upper bound on $M$, so this condition may be regarded as a weakening of Case's assumption that $f$ is bounded above when $N=\infty$. Finally, this theorem is proved in \cite{Case} when $N\in (n,\infty)$ with no assumption on $f$, so we will concern ourselves only with $N\in [-\infty,1]$ (which, as we note, includes $N=\infty$).

Now a natural question to consider is whether the $f$-generic condition is necessary. When $f$-terms are not present, singularities can be avoided but the geometry can be expected to exhibit rigidity. A standard result is the Cheeger-Gromoll splitting theorem in Riemannian geometry. Timelike and null versions exist in Lorentzian geometry, and Case extended the timelike theorem to the Bakry-\'Emery case with $N\in (n,\infty]$ (again with $f$ assumed to be bounded if $N=\infty$). We are able to extend the timelike splitting theorem as follows.

\begin{theorem}\label{theorem1.5}
Let $(M,g)$ be a timelike geodesically complete and $f$-geodesically complete spacetime with an $f$-complete timelike line.
\begin{itemize}
\item[(i)] If $(M,g)$ obeys $\tcd(0,N)$ for some $N\in [-\infty,1)$ then $(M,g)$ splits as a Riemannian product $ds^2=-dt^2+h$ and $f$ is independent of $t$.
\item[(ii)] If $(M,g)$ obeys $\tcd(0,1)$ then $f$ splits as a sum $f:=F(t)+G(y^{\alpha})$ and $(M,g)$ splits as a warped product $ds^2=-dt^2+e^{2F(t)/(n-1)}h$ on $M\equiv{\mathbb R}\times \Sigma\ni(t,y^{\alpha})$.
\end{itemize}
\end{theorem}

In \cite{Galloway}, Galloway gives a null splitting theorem for Lorentzian geometry. The null splitting theorem has not previously been extended at all to the Bakry-\'Emery case, not even for $N>n$. We obtain the following null splitting theorem:

\begin{theorem}\label{theorem1.6}
Let $(M,g)$ be a null geodesically complete   spacetime containing a null line $\eta$.
\begin{itemize}
\item[(i)] If $(M,g)$ obeys $\ncd(N)$ for some $N\in [-\infty,2)\cup (n,\infty]$  and  $(M,g)$ is $f$-complete when $N \in [-\infty,2)\cup \{\infty\}$ then $\eta$ is contained in a smooth closed achronal totally geodesic null hypersurface $S$ and $f$ is constant along the null generators of $S$.
\item[(ii)] If $(M,g)$ obeys $\ncd(2)$ and is $f$ complete  then $\eta$ is contained in a smooth, closed, achronal, totally umbilic null hypersurface.
\end{itemize}
\end{theorem}

If Theorem \ref{theorem1.5} and \cite{Case} are a guide, it would seem reasonable to expect that Theorem \ref{theorem1.6} may hold for $N\in (n,\infty)$ without the $f$-completeness assumption. This could potentially be proved by revisiting the maximum principle argument of \cite{Galloway}, similar to what was done in \cite{Case} for the timelike case. As the zero $f$-mean curvature condition is nonhomogeneous, one would presumably seek a generalization of Galloway's maximum principle argument to a nonhomogeneous condition on a null hypersurface. This is primarily an analytic question, with independent interest and using techniques beyond those of geodesic geometry, so we have chosen not to pursue the question here.

Any effort to make our argument self-contained would entail significant redundancy with \cite{Case} and \cite{BEE}. Therefore, we mostly limit discussion to those details where our arguments differ from those of \cite{Case}. We refer the reader to that reference for those parts of the proof that can be applied here with little or no modification. For general theoretical background, including Jacobi and Lagrange tensors and Busemann functions, see \cite{BEE}.

This paper is organized as follows. Section 2.1 establishes conditions for the existence of conjugate pairs of points along timelike geodesics subject to curvature-dimension conditions (especially with $N\in [-\infty,1]$). Section 2.2 contains similar results for null geodesics. Section 2.3 contains the proof of Theorem \ref{theorem1.4}. Section 3.1 contains estimates for Busemann support functions needed for the timelike splitting theorem \ref{theorem1.5}, whose proof is given in Section 3.2, except for part of the argument in the $N=1$ case which is postponed until the next section. Section 4 contains a discussion of weighted projective and conformal connections which arise naturally in the critical cases ($N=1$ for the timelike splitting theorem, $N=2$ for the null splitting theorem). We are then able to complete the proof of Theorem \ref{theorem1.5} when $N=1$ by showing that the local warped product splitting proved in Section 3.2 can be promoted to a global splitting. Section 5 contains the proof of Theorem \ref{theorem1.6}.

\subsection{Conventions} Our convention for the synthetic dimension is such that $N=n$ is the case of standard Lorentzian geometry. Other authors sometimes refer to $m=N-n$ as the synthetic dimension. We denote the limit $N\to\infty$ by writing $N=\infty$. There is no distinction between this limit and the limit $N\to -\infty$, so we regard $N$ as if it were valued on a line compactified at infinity, and often denote the infinite $N$ limit by $N=-\infty$. When we state that a theorem is valid for, say, $(n,\infty]$ or $[-\infty,1)$, we mean that the limit of infinite $N$ is included.

\subsection{Acknowledgements} The work of EW was supported by a Discovery Grant RGPIN 203614 from the Natural Sciences and Engineering Research Council (NSERC). The work of WW was supported by a grant from the Simons Foundation (\#355608, William Wylie) and a grant from the National Science Foundation (DMS-1654034).

\section{Conjugate points}
\setcounter{equation}{0}

\noindent In this section, we adapt the arguments of \cite[Section 3]{Case} to our assumptions when $N=\infty$ and when $N\le 1$ for the timelike splitting theorem and to $N\le 2$ for the null splitting theorem. Almost all of this section is standard textbook material (see, e.g., \cite{BEE}) generalized in \cite{Case, GW, EWW} to include $f$-terms. We further modify extant results where necessary to account for the replacement of the usual boundedness condition on $f$ with our milder condition of $f$-completeness when $N\in [-\infty,1]$, but will avoid unnecessarily repeating derivations that appear elsewhere. In what follows, $A$ is a Jacobi tensor along a future-timelike or future-null geodesic $\gamma$. For background on Lagrange and Jacobi tensors, see \cite{BEE}.

\subsection{Conjugate pairs along timelike geodesics}
Let the spacetime dimension be $n\ge 2$. For $\gamma$ a future-timelike geodesic, let $\gamma'(t)=\frac{d}{dt}\gamma$ where $t$ denotes a proper time parameter. Our starting point is the $f$-\emph{Raychaudhuri equation} \cite[Proposition 2.9]{Case} governing the expansion scalar $\theta$. We begin with a few definitions. The $f$-\emph{expansion scalar} $\theta_f$, defined in terms of the usual expansion scalar $\theta=A'A^{-1}$ for $A$ a Jacobi tensor along $\gamma$, is
\begin{equation}
\label{eq2.1}
\theta_f:=\theta-\nabla_{\gamma'}f=:\theta-f'\ ,
\end{equation}
where we abbreviate $f'(t):=(f\circ \gamma)'(t)$. In fact, $\theta_f$ is the trace of the endomorphism $B_f$ which, for a Jacobi tensor $A$ along timelike geodesics, is
\begin{equation}
\label{eq2.2}
B_f=A'A^{-1}-\frac{1}{(n-1)}\left ( \nabla_{\gamma'}f\right )\id \ .
\end{equation}
Here $\id$ is the identity on the orthogonal complement to $\gamma'(t)$. (It can be convenient to regard $B_f$ as a tensor on $M$, and then $\id$ is the projector into the orthogonal complement of $\gamma'(t)$.) Then Case shows that $\theta_f$ obeys the \emph{Raychaudhuri equation} (see also \cite[equation (2.5)]{EWW}), which in the \emph{vorticity-free} case is
\begin{equation}
\label{eq2.3}
\theta_f'=-\ric_f^N(\gamma',\gamma')-\tr \sigma_f^2-\frac{1}{(n-1)}\left [ \theta_f^2+2\theta_f(f\circ \gamma)'+\frac{(1-N)}{(n-N)}(f\circ \gamma)'{}^2 \right ]\ ,
\end{equation}
where the shear $\sigma=\sigma_f$ is the tracefree part of $B_f$. We note that \cite[Proposition 2.9]{Case} includes non-zero vorticity $\omega$, but we need only consider the vorticity-free case here. If the vorticity vanishes at any point along $\gamma$ then it vanishes at every point along $\gamma$. It is convenient to normalize $\theta_f$ by writing
\begin{equation}
\label{eq2.4}
x_f:=\frac{\theta_f}{n-1}\ .
\end{equation}
Then \eqref{eq2.3} is
\begin{equation}
\label{eq2.5}
x_f'=-\frac{1}{(n-1)}\left [\ric_f^N(\gamma',\gamma')+\tr \sigma_f^2\right ]-x_f^2-\frac{2x_f (f\circ \gamma)'}{(n-1)}-\frac{(N-1)(f\circ \gamma)'{}^2}{(N-n)(n-1)^2}\ .
\end{equation}

The strategy here is that $\theta_f:=\theta -\nabla_{\gamma'}f :=\tr \left ( A'A^{-1}\right )-\nabla_{\gamma'}f =\frac{(\det A)'}{\det A}-\nabla_{\gamma'}f$, so that if $\left \vert \theta_f\right \vert \to\infty$ as $t\to b$, since $f$ is differentiable then $|\theta|\to \infty$ and $A$ must be degenerate at $t=b$. We will use this to find conjugate points along $\gamma$.

Case shows that $B_f$ evolves along the geodesic $\gamma$ so as to obey the usual matrix Riccati equation modified by $f$-terms \cite[Proposition 2.8]{Case}, namely
\begin{equation}
\label{eq2.6}
\begin{split}
B_f'=&\, -R_f-B_f^2-\frac{2f'}{(n-1)}B_f\ ,\\
R_f:=&\, \riem(\cdot,\gamma')\gamma'+\frac{1}{(n-1)}\left ( \hess f(\gamma',\gamma') +\frac{f'^2}{(n-1)}\right ) \id\ .
\end{split}
\end{equation}
To fix conventions, $\riem$ is as given by \cite[equation 2.10]{BEE}. A simple calculation shows that
\begin{equation}
\label{eq2.7}
\tr R_f = \ric_f^N(\gamma',\gamma')+\frac{(1-N)}{(n-N)(n-1)}f'^2\ .
\end{equation}

\begin{definition}[\cite{Case}, Definition 3.1] \label{definition2.1} A timelike geodesic satisfies the \emph{$f$-generic condition} if $R_f$ is nonzero somewhere along it. A spacetime satisfies the \emph{timelike $f$-generic condition} if every inextendible timelike geodesic satisfies the $f$-generic condition.
\end{definition}

We see from (\ref{eq2.7}) that if $N\le 1$ or $N>n$ (including $N=\infty$) then the timelike $f$-generic condition will hold provided $\ric_f^N(\gamma',\gamma')>0$ somewhere along each complete timelike geodesic.

\begin{lemma}\label{lemma2.2} Let $\dim M=n\ge 2$ and let $N\in [-\infty,1]\cup (n,\infty]$. Let $\gamma:{\mathbb R}\rightarrow M$ be an inextendible timelike geodesic; if $N\in [-\infty,1]$ then further require that $\gamma$ be $f$-complete. Assume that $TCD(0,N)$ holds along $\gamma$ and that $\gamma$ has a point $\gamma(t_1)$ such that $R_f(t_1)\neq 0$. Then either $\gamma$ has a conjugate pair of points, or $\gamma$ is incomplete.
\end{lemma}

This is a modified version of \cite[Proposition 3.4]{Case}. For $n<N< \infty$ the proof given in \cite{Case} suffices, but for $N=\infty$ Case's proof has a stronger condition on $f$ than what we assume. Since $N=\infty$ and $N=-\infty$ are the same here, we can restrict attention to $N=[-\infty,1]$. We need the following result, which is a focusing lemma established in \cite{EWW} which extends Case's Lemma 3.3 to $N=[-\infty,1]$ with the $f$-completeness condition.

\begin{lemma}\label{lemma2.3}
Let $\gamma$ be a future complete and future $f$-complete timelike geodesic along which $\tcd(0,N)$ holds for some $N\in [-\infty,1]$. If there is a $t_0$ in the domain of $\gamma$ such that $\theta_f(t_0)<0$ then there is a point conjugate to $\gamma(t_0)$ along $\gamma$.
\end{lemma}

\begin{proof}
Under these conditions, \cite[Lemma 2.2]{EWW} can be applied, from which we may conclude that $x_f\to -\infty$ as $t\nearrow t_1$ for some $t_1\in \left [ t_0, t_0 + \frac{1}{\theta_f(t_0)}e^{-\frac{2f(t_0)}{(n-1)}}\right ]$. Then we observe that $\theta_f\to -\infty \implies \theta \to -\infty$ since $f$ is differentiable.
\end{proof}

\begin{definition}[\emph{cf} {\cite[Definition 3.8]{Case}}]\label{definition2.4}
Let $\gamma$ be a complete timelike geodesic and say that $R_f(\gamma(t_1)):=R_f(t_1)\neq 0$. Let ${\mathcal A}$ be the set of all Jacobi tensor fields $A$ along $\gamma$ such that $A(t_1):=A(\gamma(t_1))=\id$. We define $L_+:=\{ A\in {\mathcal A}:\theta_f(t_1)=\tr A'(t_1)\ge 0 \}$ and $L_-:=\{ A\in {\mathcal A}:\theta_f(t_1)=\tr A'(t_1)\le 0 \}$.
\end{definition}

\begin{lemma}[\emph{cf} {\cite[Lemma 3.9]{Case}}]\label{lemma2.5}
Let $\gamma$ be a complete and $f$-complete timelike geodesic such that $R_f(\gamma(t_1)):=R_f(t_1)\neq 0$. Assume that $\tcd(0,N)$ holds for some fixed $N\in [-\infty,1]$. Then for each $A\in L_-$ there is a number $t_2>t_1$ such that $\det A (t_2)=0$, and for each $A\in L_+$ there is a number $t_0<t_1$ such that $\det A (t_0)=0$.
\end{lemma}

\begin{proof}
The proof is exactly as given for \cite[Lemma 3.9]{Case} except that it relies on Lemma \ref{lemma2.3} above in instead of \cite[Lemma 3.3]{Case}. \end{proof}

Then the proof of Lemma \ref{lemma2.2} follows, just as the proof of \cite[Proposition 3.4]{Case} follows from \cite[Lemma 3.9]{Case}. Furthermore, we have the following corollary of Lemma \ref{lemma2.2}.

\begin{corollary}\label{corollary2.6}
Let $\dim M=n\ge 2$ obey the $f$-completeness condition, let the timelike $f$-generic condition hold, and let $TCD(0,N)$ hold for some fixed $N\in [-\infty,1]$. Then each complete timelike geodesic has a pair of conjugate points.
\end{corollary}

\subsection{Conjugate pairs along null geodesics}
Now we must take the spacetime dimension to be $n\ge 3$. For $\gamma$ a future-null geodesic, let $\gamma'(t)=\frac{d}{dt}\gamma$ where $t$ denotes an affine parameter. The orthogonal complement to $\gamma'$ now contains $\gamma'$, but variations along $\gamma'$ can be absorbed by the parametrization, so we quotient out by $\gamma'$ as discussed in \cite[Section 10.3]{BEE}. The rank of Jacobi tensors ${\bar A}$ at generic points is then $n-2$; we use an overhead bar to recall the reduced rank. We let ${\bar \theta}$ denote the expansion scalar for ${\bar A}$, ${\bar x}_f$ denote the normalized expansion scalar, and ${\bar \sigma}\equiv {\bar \sigma}_f$ denote the corresponding shear, so that
\begin{equation}
\label{eq2.8}
\begin{split}
{\bar B}_f:=&\, {\bar A}'{\bar A}^{-1}-\frac{1}{(n-2)}\left ( \nabla_{\gamma'}f\right ){\bar \id} \ , \\
{\bar \sigma}:=&{\bar \sigma}_f:=h(\cdot, {\bar B}_f)-\frac{1}{(n-2)}\theta_f h\ , \\
{\bar \theta}_f:=&\, {\bar \theta}-\nabla_{\gamma'}f=\tr {\bar B}_f\ ,\\
{\bar x}_f:=&\, \frac{{\bar \theta}_f}{n-2}\ ,
\end{split}
\end{equation}
with ${\bar \id}$ the identity on $\left [ \gamma'^{\perp} \right ]$, the orthogonal complement of $\gamma'$ quotiented by $\gamma'$, and $h$ the induced metric. Then the Raychaudhuri equation for null geodesics is
\begin{equation}
\label{eq2.9}
\begin{split}
{\bar \theta}' =&\, -\ric(\gamma',\gamma')-{\bar \sigma}^2-\frac{{\bar \theta}^2}{n-2}\\
\implies {\bar \theta}_f'=&\, -\ric_f^N(\gamma',\gamma')-\tr {\bar \sigma}_f^2 -\frac{{\bar \theta}_f^2}{n-2} -\frac{2(f\circ\gamma)'}{(n-2)}{\bar \theta}_f-\frac{(N-2)\left ( (f\circ \gamma )'\right )^2}{(N-n)(n-2)}\\
\implies {\bar x}_f'=&\, -\frac{1}{(n-2)}\left (  \ric_f^N(\gamma',\gamma')+\tr {\bar \sigma}_f^2\right ) -{\bar x}_f^2  -\frac{2{\bar x}_f f'}{(n-2)}-\frac{(N-2)}{(N-n)(n-2)^2} (f')^2\ ,
\end{split}
\end{equation}
The evolution equation for ${\bar B}_f$ is
\begin{equation}
\label{eq2.10}
\begin{split}
{\bar B}_f'=&\, -{\bar R}_f-{\bar B}_f^2-\frac{2f'}{(n-2)}{\bar B}_f\ ,\\
{\bar R}_f:=&\, \riem(\cdot,\gamma')\gamma'+\frac{1}{(n-2)}\left ( \hess f(\gamma',\gamma') +\frac{f'^2}{(n-2)}\right ) \id\ .
\end{split}
\end{equation}
The trace of the second equation in \eqref{eq2.10} yields
\begin{equation}
\label{eq2.11}
\begin{split}
\tr {\bar R}_f = &\, \ric (\gamma',\gamma')+ \hess (\gamma',\gamma')+\frac{f'^2}{(n-2)}\\
=&\, \ric_f^N+\frac{(2-N)}{(n-N)(n-2)}f'^2\ .
\end{split}
\end{equation}
The trace on the left-hand side is taken over the quotient space; see \cite[Proposition 2.12] {BEE} for calculational details.

In analogy to Definition \ref{definition2.1} we now define
\begin{definition}[\cite{Case}, Definition 3.1] \label{definition2.7} A null geodesic satisfies the \emph{$f$-generic condition} if ${\bar R}_f$ is nonzero somewhere along it. A spacetime satisfies the \emph{null $f$-generic condition} if every inextendible null geodesic satisfies the $f$-generic condition. If a spacetime obeys both the timelike and null generic conditions, we say that the spacetime obeys the \emph{generic curvature condition}.
\end{definition}

We see from \eqref{eq2.11} that if $N\le 2$ or $N>n$ (including $N=\infty$) then the null $f$-generic condition will hold provided $\ric_f^N(\gamma',\gamma')>0$ somewhere along each complete null geodesic (\emph{cf} \cite[Proposition 2.12] {BEE}).

It is well-known that, in the absence of $f$-terms, the equations governing timelike geodesics map to those governing null geodesics under the replacement $n\mapsto n-1$, and that modulo this replacement the analysis of the null Raychaudhuri equation follows precisely as it does for the timelike Raychaudhuri equation. We see from the above equations that the same is true in the $f$-Bakry-\'Emery case provided we also make the replacement $N\mapsto N-1$. Lemmata \ref{lemma2.2} and \ref{lemma2.5} carry over, \emph{mutatis mutandis}, with $n\ge 3$ now and $N\in [-\infty,2]$, as does Corollary \ref{corollary2.6}, which reads

\begin{lemma}\label{lemma2.8}
Let $\dim M=n\ge 3$ obey the $f$-completeness condition, let the null $f$-generic condition hold, and let $TCD(0,N)$ hold for some fixed $N\in [-\infty,2]$. Then each complete null geodesic has a pair of conjugate points.
\end{lemma}

\subsection{Singularity theorems}
Before proving Theorem \ref{theorem1.4}, we first extend Case's singularity Theorem, \cite[Theorem 4.4]{Case},  to negative $N$, while relaxing his boundedness assumption on $f$ when $N=\infty$. If a spacetime is chronological (i.e., has no closed timelike curves) and if every inextendible null geodesic has a conjugate pair, then the spacetime is strongly causal (every point has a neighborhood to which no nonspacelike geodesic beginning that point, having exited, returns) \cite[Theorem 12.39]{BEE}.

We recall that a spacetime is causally disconnected if it contains a compact set $K$ and sequences $p_n$ and $q_n\in I^+(p_n)$ diverging to infinity (i.e., escaping any compact set as $n$ increases) such that every future-causal curve from $p_n$ to $q_n$ intersects $K$. A \emph{chronological} spacetime is one with no closed timelike curves.

\begin{theorem}\label{theorem2.9}
Let $(M,g)$, $\dim M = n \ge 3$, be a chronological spacetime which is causally disconnected and satisfies the $f$-generic and $\tcd(0,N)$ conditions, $N\in [-\infty,1]$. Then $(M,g)$ is nonspacelike geodesically incomplete.
\end{theorem}
\begin{proof}
Because $\tcd(0,N)$ holds, so does $\ncd(N)$. By Lemma \ref{lemma2.8}, every complete null geodesic has a conjugate pair. Then either the spacetime is strongly causal or it contains an incomplete null geodesic. But every strongly causal, causally disconnected spacetime has a nonspacelike line \cite[Theorem 8.13]{BEE}, which necessarily has no conjugate points and which cannot be complete as it would violate Corollary \ref{corollary2.6} or Lemma \ref{lemma2.8}.\footnote
{At one point in the discussion, \cite{Case} cites stable causality, but \cite[Theorem 8.13]{BEE} requires only strong causality (a strictly weaker condition).}
\end{proof}

To prove Theorem \ref{theorem1.4}, we need a focusing lemma for null geodesics orthogonal to a codimension-2 spacelike hypersurface $\Sigma$. To that end, let $p\in \Sigma$ and define the second fundamental form $K:T^{\perp}_p\Sigma\times T_p\Sigma\times T_p\Sigma:(\nu,x,y)\mapsto -g(\nu,\nabla_X Y)$, where $X$ and $Y$ are smooth extensions of $x$ and $y$ to a neighborhood $U$ of $p$. If $\nu$ is actually the tangent field (in $U$) to a congruence of null geodesics including $\beta$, then the Leibniz rule yields $B=K(\beta',\cdot,\cdot)$. The associated null mean curvature is $\tr_h K=:\theta$, with $h$ the induced metric on $\Sigma$.

\begin{lemma}[\emph{cf} {\cite[Lemma 4.10]{Case}}]\label{lemma2.10}
Let $(M,g)$ be a spacetime with $\dim M =:n >3$, and let $\beta :J\to M$ be an inextendible null geodesic that meets $\Sigma$ orthogonally at $p=\beta(t_0)$. Assume that $\ncd(N)$ holds along $\beta$ with $N\in [-\infty,2]\cup (n,\infty]$. If $N\in [-\infty,2]$, suppose as well that $(M,g)$ obeys the $f$-completeness condition. Let $\theta(t_0)$ be the null mean curvature of $\Sigma$ defined along the $\beta$ congruence at $p$ and let $\theta_f:=\theta - \nabla_{\beta'}f$. If $\theta_f=-\delta<0$ at $p$, then there is a focal point to $\Sigma$ along $\beta$ at some $t\in[t_0,t_0+\frac{1}{\delta}\epsilon_N]$ where
\begin{equation}\label{eq2.12}
\epsilon_N:=\begin{cases}1,& N\in (n,\infty)\\ e^{-2f(p)/(n-2)},& N\in [-\infty,2]\end{cases}
\end{equation}
provided $\beta$ can be extended far enough to the future. Dually, if $\theta_f=a>0$ at $p$, then there is a focal point to $\Sigma$ along $\beta$ at some $t\in[t_0-\frac{1}{\delta}\epsilon_N,t_0]$, provided $\beta$ can be extended far enough to the past.
\end{lemma}

\begin{proof}
For $N\in (n,\infty)$, this is proved as \cite[Lemma 4.10]{Case}. Beware the sign convention for the second fundamental form used there is that of \cite{BEE}, and differs from ours (our corresponds to that of \cite{Galloway}).

For $N\in [-\infty,2]$, the proof is that given in the timelike case in \cite[Lemma 2.2]{EWW} with $n$ replaced by $n-1$ and $N$ replaced by $N-1$.
\end{proof}

Now we follow a well-worn path. In the proof below, we apply Lemma \ref{lemma2.10} with $N\in [-\infty,1]$ (we do not need that the Lemma also holds for $N\in (1,2]$ here).

\begin{proof}[Proof of Theorem \ref{theorem1.4}]
Say that $(M,g)$ contains a closed $f$-trapped surface $S$, a closed codimension 2 spacelike surface such that, for both null geodesic congruences that leave it orthogonally, the $f$-modified expansion scalars $\theta_f$ are both negative (or both positive). Then the expansion scalars are bounded away from zero on this surface, so every complete null geodesic in these congruences will have a focal point within a uniformly bounded Lorentzian distance to the future (to the past if the scalars are positive). Then either at least one of these geodesics is incomplete, or the future (or past) of $S$ is compact and $S$ is a \emph{trapped set}. This is \cite[Proposition 4.9]{Case}. Then we have a chronological spacetime in which, by Corollary \ref{corollary2.6} and Lemma \ref{lemma2.8}, every complete nonspacelike geodesic has a conjugate pair, and which contains a trapped set. But by \cite[Theorem 12.43]{BEE} (or \cite[Theorem 4.6]{Case}), $(M,g)$ must contain an incomplete nonspacelike geodesic. This proves Theorem \ref{theorem1.4} under assumption (1) of the theorem, that $(M,g)$ contains a closed $f$-trapped set.

To prove that assumption (2) yields the theorem, consider a null geodesic $\beta$ with initial endpoint $f$-reconverging at some $p=\gamma(t_0)$ to the future (say; a dual argument works to the past) of $\gamma(0)$. Recall that a future-null geodesic $\beta:[0,b)\to M$, $\gamma(0)=p$ is \emph{$f$-reconverging} at $\beta(t_0)$, $t_0\in [0,b)$, if there is a Lagrange field ${\bar A}$ along $\beta$ with ${\bar A}(0)=0$, ${\bar A}'(0)=\id$, such that the associated $f$-expansion scalar obeys $\theta_f(t_0)<0$. But by Lemma \ref{lemma2.8}, if the geodesic is future-complete, there will be a point along it conjugate to $p$. Now since the space of future directions at $p$ is compact, this implies that the future boundary of $p$ is compact, and so $p$ is a trapped set. As above, this and \cite[Theorem 12.43]{BEE} (or \cite[Theorem 4.6]{Case}) together imply the existence of an incomplete geodesic.

Finally, that assumption (3) implies the theorem follows from remarks in \cite[pp 471--472]{BEE}, where it is argued that a closed surface $S$, if achronal, must be its own future boundary, and since it is also compact, it is therefore future-trapped, and then incompleteness follows as above. If $S$ is not achronal, one can pass to a Lorentzian covering space in which the lift is achronal and thus future-trapped, implying that the covering spacetime is incomplete. But then the original spacetime is incomplete as well.
\end{proof}

\section{The timelike splitting theorem for $N\in [-\infty,1]$}
\setcounter{equation}{0}

\subsection{Maximum principle}
In this section, we adapt the arguments of \cite[Section 5]{Case}. We consider in particular \cite[Lemma 5.5]{Case}, which is a computation based on the second variation of arclength along a timelike geodesic. In what follows, $d$ indicates Lorentzian distance; i.e., the supremum of proper time along all timelike curves joining two points. We recall the following definition.

\begin{definition}\label{definition3.1}
Let $S$ be a subset of $M$. A future-inextendible nonspacelike geodesic $\alpha:[0,a]\to M$ is a \emph{future $S$-ray} if $d(S,\alpha(t))=t$ for all $t\in [0,a)$. An $\alpha(0)$-ray is simply called a \emph{ray}. Past-directed $S$-rays are defined dually.
\end{definition}

In particular, future-timelike rays maximize the Lorentzian distance between any two of their points.

\begin{lemma}[\emph{cf} {\cite[Lemma 5.5]{Case}}]\label{lemma3.2}
Let $(M,g)$ obey $\tcd(0,N)$ for $N\in [-\infty,1]$ and let $\alpha:[0,\infty)\to M$ be a ray. Define $d_r(x):=d(x,\alpha(r))$ be the Lorentzian distance from $x$ to $\alpha(r)$ and let $\sigma:[0,\rho]\to M$ be a past-directed maximal timelike geodesic from $\alpha(r)=\sigma(0)$ to $q=\sigma(\rho)\in I^-(\alpha(r))$. Then
\begin{equation}
\label{eq3.1}
\begin{split}
\Delta_f d_r(q):=&\, \Delta  d_r(q) - g(\nabla f, \nabla d_r(q))\\
\ge &\, -\frac{(n-1)e^{\frac{-2f(q)}{n-1}}}{s (\rho)}\ ,
\end{split}
\end{equation}
where $s (\rho) = \int \limits_0^{\rho} e^{\frac{-2f(\sigma(t))}{n-1}}dt$.
\end{lemma}

\begin{proof}
Following the proof in \cite{Case}, we use the maximality of $\sigma$ and the second variation formula for arclength for a variation through geodesics based about $\sigma$ with variation vector field $v$ to write
\begin{equation}
\label{eq3.2}
\begin{split}
0\ge L''(0) &= -g( \sigma',\nabla_vv)\big\vert_0^{\rho} -I(v,v) \\
I(v,v) &=\int\limits_0^{\rho} \left [ g(v',v')-g\left ( \riem(v,\sigma')\sigma',v\right ) \right ] dt\ ,
\end{split}
\end{equation}
where $I(v,v)$ is the index form and $(\, )':=\frac{d}{dt}(\, )$ denotes differentiation with respect to $t$. Re-parametrize the geodesic $\sigma$ with the parameter $s(t_0) = \int \limits_0^{t_0}  e^{\frac{-2f(\sigma(t))}{n-1}}dt $ and let $( \dot{\,}):=\frac{d}{ds}(\,)$ denote differentiation with respect to the parameter $s$.

As is shown in \cite[Proposition 4.1]{KWY} in the Riemannian case,  the index form can be given a very clean formula when written in terms of the parameter $s$. For a variation field $w$, consider
\begin{equation}
\label{eq3.3}
\begin{split}
I\left(e^{\frac{f}{n-1}} w, e^{\frac{f}{n-1}} w\right)  &= \int\limits_0^{\rho} \left [  g\left(\left(e^{\frac{f}{n-1}} w\right)', \left(e^{\frac{f}{n-1}} w\right)'\right) - e^{\frac{2f}{n-1}} g\left ( \riem(w,\sigma')\sigma',w\right ) \right ] dt  \\
\end{split}
\end{equation}
Expanding the first term in (\ref{eq3.3}) and integrating by parts  gives
\begin{equation}
\label{eq3.4}
\begin{split}
&\, \int\limits_0^{\rho} g\left(\left(e^{\frac{f}{n-1}} w\right)', \left(e^{\frac{f}{n-1}} w\right)'\right) dt\\
=&\, \int\limits_0^{\rho} \left[  e^{\frac{2f}{n-1}}g(w', w')  + \frac12 \left(e^{\frac{2f}{n-1}}\right)' \left(g(w,w)\right)' + e^{\frac{2f}{n-1}}\left(\frac{f'}{n-1}\right)^2g(w,w) \right] dt \\
=&\,  \frac12 \left(e^{\frac{2f}{n-1}}\right)' g(w,w)\big \vert_0^{\rho} + \int\limits_0^{\rho}  e^{\frac{2f}{n-1}} \left[ g(w', w') - \frac{f''}{n-1} - \left( \frac{f'}{n-1}\right)^2  \right]dt
\end{split}
\end{equation}
Plugging \eqref{eq3.4} back into \eqref{eq3.3} gives us
\begin{equation}
\label{eq3.5}
\begin{split}
I\left(e^{\frac{f}{n-1}} w, e^{\frac{f}{n-1}} w\right)  &=  \frac12 \left(e^{\frac{2f}{n-1}}\right)' g(w,w)\big \vert_0^{\rho} + \int\limits_0^{\rho}  e^{\frac{2f}{n-1}} \left[ g(w', w') - R_f(w,w)  \right]dt\\
&= \frac{\dot{f}}{n-1}g(w,w) \big \vert_0^{\rho} + \int\limits_0^{s(\rho)}\left (  g( \dot{w}, \dot{w})  - e^{\frac{4f}{n-1}}R_f(w,w) \right ) ds\ ,
\end{split}
\end{equation}
where $R_f$ is the $(0,2)$-tensor obtained from the endomorphism defined in \eqref{eq2.6} by lowering an index.

Choose an orthonormal basis $\{ e_{(1)},\dots,e_{(n-1)},e_{(n)}=\sigma'(0) \}$ at $\sigma(0)$ and extend it to a neighborhood of $\sigma$ by parallel transport. For each $i$, let $v_i=\frac{s}{s(\rho)}e_{(i)}$ and $w_i = e^{\frac{-f(\sigma(\rho))}{n-1}} v_i = e^{\frac{-f(\sigma(\rho))}{n-1}}\frac{s}{s(\rho)} e_i$. Plug these into \eqref{eq3.2} and \eqref{eq3.5} and sum over $i$ to obtain
\begin{equation}
\label{eq3.6}
\begin{split}
0 \ge &\, -\Delta  d_r(q) - \sum_{i=1}^{n-1} I(v_i, v_i) \\
= &\, -\Delta  d_r(q) - \sum_{i=1}^{n-1}I\left(e^{\frac{f}{n-1}} w_i, e^{\frac{f}{n-1}} w_i\right) \\
= &\, -\Delta  d_r(q) - \sum_{i=1}^{n-1} \left [ \frac{s^2 f'}{s^2(\rho)(n-1)}g(e_i,e_i) \big \vert_0^{\rho} + \int\limits_0^{s(\rho)}\left (  g( \dot{w}_i, \dot{w}_i)  - e^{\frac{4f}{n-1}}R_f(w_i,w_i) \right ) ds \right ] \\
= &\, -\Delta  d_r(q) -  \nabla_{\sigma'} f(\rho) -\frac{(n-1)e^{\frac{-2f(\sigma(\rho))}{n-1}}}{s(\rho)}\\
&\, + \int \limits_0^{s(\rho)} e^{\frac{4f}{n-1}} \left [ \ric_f^N(\sigma',\sigma') +\frac{(1-N)f'{}^2}{(n-N)(n-1)}\right ] ds\\
\geq &\,  -\Delta  d_r(q) -  \nabla_{\sigma'} f(\rho) -\frac{(n-1)e^{\frac{-2f(\sigma(\rho))}{n-1}}}{s(\rho)}\ ,
\end{split}
\end{equation}
using condition $\tcd(0,N)$ with $N\in [-\infty,1]$.
Using that $\nabla_{\sigma'}f(\rho)=g(\sigma',\nabla f)=-g(\nabla d_r,\nabla f)$, then this implies that
\begin{equation}
\label{eq3.7}
\Delta_f  d_r(q) \ge -\frac{(n-1)e^{\frac{-2f(q)}{n-1}}}{s(\rho)}\\
\end{equation}
as claimed.
\end{proof}

We use this estimate to extend the maximum principle for the Busemann functions to the $TCD(0,N)$ condition for $N \leq 1$.

First recall the definition of a timelike Busemann function and associated upper support function. We give only basic definitions; for details see \cite[Section 14.2]{BEE} or \cite{GH}. Given a future-timelike ray $\gamma:[0,\infty)\to M$ parametrized by proper time (i.e., unit speed), the \emph{Busemann function} $b:M\to{\mathbb R}$ is defined by
\begin{equation}
\label{eq3.8}
b(q):=b_{\gamma}^+(q):=\lim_{t\to\infty} b_t(q) := \lim_{t\to\infty} \left ( t - d(q,\gamma(t)\right )\ .
\end{equation}
Busemann functions are not necessarily differentiable, so it is helpful to define smooth support functions.
To do this, one first considers an \emph{asymptote} $\alpha:[0,\infty)\to M$ to $\gamma$ beginning at some $q=\alpha(0) \in M$. This is the limit curve of a sequence of maximal timelike geodesics that each begin at $q$ and end at $\gamma(t_n)$, where $n$ indexes the sequence and $t_n\to\infty$. More generally, if the initial endpoints are not all $q$ but are instead a sequence $q_n\to q$, the limit curve $\alpha$ is called a \emph{generalized co-ray} (all asymptotes are generalized co-rays). The \emph{generalized co-ray condition} holds at $q$ if, for $\gamma$ a future-timelike $S$-ray and $q\in I^+(S)\cap I^-(\gamma)$, every generalized co-ray from $q$ to $\gamma$ is timelike. Finally, \emph{upper support functions} $b_{q,t}(x)$ are defined by
\begin{equation}
\label{eq3.9}
b_{p,t}(x) = b(p) + t - d(x, \alpha(t))\ .
\end{equation}

As in \cite{Case} we let $H_{f, \Sigma} = H_{\Sigma} - g(\nabla f, \nu)$, where $\nu$ is the future pointing unit normal along $\Sigma$, and we are using the sign convention $H_{\Sigma} = \divergence \nu =\nabla^i \nu_i$.

\begin{theorem}[\emph{cf} {\cite[Theorem 5.7]{Case}}]\label{theorem3.3}
Let $(M,g)$ be a future timelike  geodesically complete and $f$-complete space-time which
obeys the  $\tcd(0,N)$ condition  for $N\in [-\infty,1]$  and let $\gamma$ be a timelike future S-ray. Let $W \subset I^{-}(\gamma) \cap I^{+}(S)$ be an open set
on which the generalized timelike co-ray condition holds. Let $\Sigma \subset W$  be a connected
smooth spacelike hypersurface with nonpositive $f$-mean curvature $H_{f, \Sigma} \leq 0$. If the
Busemann function $b = b^+_{\gamma}$ attains a minimum along $\Sigma$, then $b$ is constant along
$\Sigma$.
\end{theorem}

The proof will follow exactly along the lines of the proof given in \cite[Theorem 5.7]{Case}. There Case establishes that, for $\alpha: [0, \infty) \rightarrow M$ a timelike asymptote to $\gamma$ at a point $p$, the upper support function $b_{p,t}(x)$ defined by \eqref{eq3.9} satisfies
\begin{equation}
\label{eq3.10}
\Delta_f b_{p,t}\le \begin{cases} (N-1)/t,& N>n\\ \frac{(n-1)}{t} -\frac{2}{t}f(p)+\frac{2}{t^2}\int_0^t f\circ\sigma.& N=\infty\end{cases}
\end{equation}
He then takes the limit $t\to\infty$ and uses that the limit is nonnegative, modulo an error which can be dominated by a negative term in a subsequent step of the calculation.

We will instead prove that, for $N=[-\infty,1]$, using our assumption of $f$-completeness, then
\begin{equation}
\label{eq3.11}
\limsup_{t \rightarrow \infty} \left(\Delta_f b_{p,t}\right)(p)  \le 0\ .
\end{equation}
This replaces the estimates \eqref{eq3.10}, and then the remainder of Case's proof goes through. This overlaps with Case's $m=\infty$ result, which we characterize as $N=-\infty$. Thus we obtain the necessary result in this case from $f$-completeness, without needing Case's assumption that $f$ is bounded above.

\begin{proof}
From Lemma \ref{lemma3.2} we have that
\begin{equation}
\label{eq3.12}
\left(\Delta_f b_{p,t}\right)(p) \le \frac{(n-1)e^{\frac{-2f(p)}{n-1}}}{s_t }\ ,
\end{equation}
where $\sigma_t$  is a unit speed  past-directed maximal timelike geodesic from $\sigma_t(0)=\alpha(t)$ to $\sigma_t(\rho_t) = p$ and $s _t = \int \limits_0^{\rho_t} e^{\frac{-2f(\sigma_t(\tau))}{n-1}}d\tau $.

Equation \eqref{eq3.11} follows if  $\lim_{t \rightarrow \infty} s_t = \infty$. Suppose not. Then we have a sequence $t_i\rightarrow \infty $ such that $s_{t_i} \le A$ for some constant $A$. Consider the sequence of unit vectors $-\sigma'_{t_i}(\rho_{t_i})$ at $p$ (note that $p\equiv \sigma_{t_i}(\rho_{t_i})$). A subsequence converges to a timelike vector $u$ at $p$ and we obtain a timelike future-directed ray $\beta$ with $\beta(0)=p$ and $\beta'(0)=u$. A subsequence of the geodesics $\sigma_{t_i}$, parametrized with the opposite orientation, converges uniformly on compact sets to $\beta$. The condition that $s_{t_i} \le A$ then contradicts the $f$-completeness assumption of the ray $\beta$.
\end{proof}

\subsection{Proof of Theorem \ref{theorem1.5}.}
Now assume that $\gamma$ is a timelike line and that $\tcd(0,N)$ holds for $N\in [-\infty,1]$. Furthermore, assume that $f$-completeness holds along $\gamma$ in both future and past directions. By Corollary \ref{corollary2.6} the $f$-generic condition must fail. We will first seek only local splitting in a neighborhood of $\gamma$. In the cases where the splitting is a direct product splitting,  the extension to global splitting is discussed in, e.g., \cite[Section 14.4]{BEE} and the argument does not depend on the presence of $f$ or the synthetic dimension $N$. In the case of a warped product splitting,  we will discuss the extension to a global splitting in the next section.

Once the constancy of $b$ along $\Sigma$ has been established, Case shows that if $\gamma$ is in fact a line (rather than merely a ray), the argument can be run both in the future and past directions. Then the respective restrictions of $\gamma$ to the future and to the past yield rays and corresponding Busemann functions, denoted $b^{\pm}_{\gamma}$, such that $b^+=b^-=0$ along $\Sigma$. Then future- and past-timelike asymptotes to $\gamma$ can be constructed from each $x\in\Sigma$. These are focal point free and meet $\Sigma$ orthogonally, so future- and past-directed asymptotes can be joined to form timelike lines. By arguments given in \cite[Section 14.4]{BEE}, one now obtains a tubular neighborhood $U$ of $\gamma$. The normal exponential map along $\Sigma$ is a diffeomorphism onto this neighborhood, giving a foliation whose leaves are images of $\Sigma$. The timelike geodesics orthogonal to $\Sigma$ are conjugate point free. All of this reasoning is standard and does not require any assumption on the synthetic dimension nor on $f$ except for what is necessary to establish the constancy of $b$ along $\Sigma$.

The leaves have mean curvature $H$ which obeys the Raychaudhuri equation, which for $x_f:=\frac{H_f}{n-2}=\frac{1}{(n-1)}\left ( H-\nabla_{\gamma}f\right )$ is the same equation as \eqref{eq2.5} and \cite[Equation 2.7]{EWW}. We re-write it as
\begin{equation}
\label{eq3.13}
e^{\frac{-2f}{n-1}}\left ( e^{\frac{2f}{n-1}}x_f\right )'+x_f^2 = -\frac{1}{(n-1)}\left [ \ric_f^N(\gamma',\gamma')+\tr \sigma_f^2\right ]-\frac{(1-N)f'^2}{(n-N)(n-1)^2}\ ,
\end{equation}
where now we take $\gamma:=\exp_p{tv}$ to be any timelike geodesic that meets $\Sigma$ orthogonally; $p=\gamma(0)\in \Sigma$, $v=\gamma'(0)\in T^{\perp}\Sigma$. Taking $p$ to lie on any leaf of the foliation, since the congruence $\exp_p (tv)$, $v=\gamma'(0)$, has no focal points, it now follows from \cite[Lemma 2.2]{EWW} and the $\tcd(0,N)$ condition that $x_f=H_f=0$ all along the foliation of $U$. Then from \eqref{eq3.13} we have $\ric_f^N(\gamma',\gamma')=0$, $\sigma_f\equiv\sigma=0$, and either $f'=0$ or $N=1$.

For $N\in [-\infty,1)$, $f$ is then constant along $\gamma$ so $\ric_f^N(\gamma',\gamma')=0 \implies \ric(\gamma',\gamma')=0$, and $H_f=0\implies H=0$. Since $\sigma_f\equiv \sigma=0$ as well, the foliation is totally geodesic, yielding the required foliation in the tubular neighborhood of our original line $\gamma$. The metric splits as $ds^2 = -dt^2+e^{2f/(n-1)}{\hat h} = -dt^2+h$ where we may write that $h:=e^{2f/(n-1)}{\hat h}$ since $f$ is independent of $t$.

For $N=1$, we have $\ric_f^1(\gamma',\gamma')=0$, $\sigma_f\equiv\sigma=0$, and $H_f=0$. The latter implies that $H=\nabla_{\gamma'}f$. Combining this with $\sigma_f\equiv \sigma=0$, we see that the metric splits as a \emph{twisted product}
\begin{equation}
\label{eq3.14}
ds^2 = -dt^2 +e^{2f/(n-1)}{\hat h}
\end{equation}
for some metric ${\hat h}={\hat h}_{\alpha\beta}dy^{\alpha}dy^{\beta}$ on $\Sigma$ and some $f(t,y^{\alpha})$ (with $y^{\alpha}$ denoting coordinates on $\Sigma$). Then the Gauss-Codazzi-Mainardi equations yield
\begin{equation}
\label{eq3.15}
\ric\left (\frac{\partial}{\partial t},\frac{\partial}{\partial y^{\alpha}} \right ) =-\frac{(n-2)}{(n-1)} \frac{\partial H}{\partial y^{\alpha}} =-\frac{(n-2)}{(n-1)} \frac{\partial^2 f}{\partial t \partial y^{\alpha}},
\end{equation}
and a simple calculation gives
\begin{equation}
\label{eq3.16}
\hess f  \left (\frac{\partial}{\partial t},\frac{\partial}{\partial y^{\alpha}}\right )+ \frac{1}{(n-1)}\left \langle \frac{\partial}{\partial t}, df\right \rangle \left \langle \frac{\partial}{\partial y^{\alpha}}, df\right \rangle=\frac{\partial^2 f}{\partial t \partial y^{\alpha}}\ .
\end{equation}
Adding these yields
\begin{equation}
\label{eq3.17}
\ric_f^1\left (\frac{\partial}{\partial t},\frac{\partial}{\partial y^{\alpha}}\right )=\frac{1}{(n-1)}\frac{\partial^2 f}{\partial t \partial y^{\alpha}}\ .
\end{equation}
But the $\tcd(0,1)$ condition $\ric_f^1(X,X)\ge 0$  for all timelike vectors $X$ and the result above that $\ric_f^1(\gamma',\gamma')=0$ together imply that $\ric_f^1\left (\frac{\partial}{\partial t},\frac{\partial}{\partial y^{\alpha}}\right )=0$.\footnote
{To see this, consider any $(0,2)$-tensor $T$ such that $T(v,v)\ge 0\ \forall\ v$ with $g(v,v)=-1$. Let $\{ e_0,e_i\}$ be an orthonormal basis and assume that $T(e_0,e_0)=0$. Let $a(t)$, $b(t)$ take values on the unit hyperbola $-a^2+b^2=1$, so that $a(0)=1$, $a'(0)=0$, $b(0)=0$, $b'(0)=1$. Construct $w_i(t)=a(t)e_0+b(t)e_i$. Then $g(w_i,w_i)=-a^2+b^2=-1$ (no sum here) and $w_i(0)=e_0$. Also, $w_i'(0)=e_i$. Now since $g(w_i,w_i)=-1$ we have $T(w_i,w_i)\ge 0$, and since $w_i(0)=e_0$ we have $T(w_i(0),w_i(0))=T(e_0,e_0)=0$. Then $t=0$ is a critical point of $T(w_i(t),w_i(t))$. Thus $0=\frac{d}{dt}\big \vert_{t=0} \left (T(w_i,w_i)\right ) =2T(w_i(0),w_i'(0))=2T(e_0,e_i)$, proving the claim.}
Hence $f$ splits as $f(t,y^{\alpha})=F(t)+G(y^{\alpha})$. Writing the metric on the leaves $\Sigma$ as $h:=e^{2G/(n-1)}{\hat g}$, we now have the warped product splitting
\begin{equation}
\label{eq3.18}
ds^2 = -dt^2 +e^{2F(t)/(n-1)}h\ .
\end{equation}

We therefore have the claimed splittings on a tubular neighborhood $U$ of the original timelike line $\gamma$. For $N\in [-\infty,1)$ the splittings may be extended globally precisely as described in \cite[pp 557--561]{BEE}. In that case, the proof of Theorem \ref{theorem1.5} is now complete.

For the case of $N=1$, we have at this stage only a local warped product splitting. The factors in the splitting are timelike geodesics and spacelike totally geodesic hypersurfaces with respect to a projectively related connection which we describe in the next section. The arguments in \cite{BEE} can be adapted to this connection, yielding a global warped product splitting. Modulo the details of the local-to-global argument, the proof in the $N=1$ case is now also complete. However, those details make use of some technology developed in the next section, after which we can explicate the key details in the local-to-global argument.

\section{Weighted and conformal connections}
\setcounter{equation}{0}
\label{Section4}

\subsection{Definitions and properties} Let $(M,g)$ be a pseudo-Riemannian manifold with a smooth function $f$.  In this section we summarize the notion of a weighted connection for the triple $(M,g,f)$ which is projectively equivalent to the Levi-Civita connection.  Two connections are called projectively equivalent if their geodesics are the same as sets.  In the Riemannian case, this connection was investigated in \cite{WY}, however much of the basic properties hold more generally for pseudo-Riemannian spaces. In this section we review these properties.

The starting point for our weighted connection is the following observation.

\begin{proposition} \label{Prop:Universal} Given an orientable pseudo-Riemannian metric $(M,g)$ and a smooth volume form $\mu$ there is a unique torsion free linear connection which is projectively equivalent to the Levi-Civita connection and makes $\mu$ parallel.
\end{proposition}

The proof of Proposition \ref{Prop:Universal} is elementary.  First note that a result of Weyl \cite{Weyl} states that any torsion-free connection projectively equivalent to $\nabla$ is of the form $\nabla^{\alpha} =  \nabla _U V - \alpha(U)V - \alpha(V) U$ for some one-form $\alpha$.  The volume form $\mu$ can then be written as a positive function times the volume element of the metric $g$ which we denote by $d\vol_g$.  If we normalize so that $\mu = e^{-\frac{n+1}{n-1} f} d\vol_g$ for some function $f$, then we obtain that the choice  $\alpha = \frac{df}{n-1}$ makes $\mu$  parallel.

Based on this, we define the weighted connection $\nabla^f$ by the formula
\begin{equation}
\label{eq4.1}
\nabla^f_X Y = \nabla_X Y - \frac{1}{(n-1)} df(X) Y - \frac{1}{(n-1)} df(Y) X.
\end{equation}
We note that $\nabla^f$ depends not only on $f$ but on $g$ as well. However, since we will always think of the background metric $g$ as being fixed, we will not emphasize this dependence.  We also see that this definition works in the case where the manifold is non-orientable, even though there is no global volume form.  The  connection $\nabla^f$ will make the locally defined volume form $e^{-\frac{n+1}{n-1} f} d\vol_g$ parallel.

The curvature tensor of $\nabla^f$ is
\begin{equation}
\label{eq4.2}
\begin{split}
R^{\nabla^f}(X,Y)Z = &\, R(X,Y)Z + \frac{1}{(n-1)}\hess(f)(Y,Z) X - \frac{1}{(n-1)} \hess(f)(X,Z) Y\\
&\, + \frac{1}{(n-1)^2}df(Y)df(Z) X - \frac{1}{(n-1)^2}df(X)df(Z) Y.
\end{split}
\end{equation}
In particular,
\begin{equation}
\label{eq4.3}
\begin{split}
\ric^{\nabla^f}(Y,Z) =&\, \ric(Y,Z) +  \hess f (Y,Z) + \frac{1}{(n-1)} df(Y)df(Z)\ , \\
=&\, \ric_f^1(Y,Z)\ .
\end{split}
\end{equation}
This shows that the Bakry-\'Emery geometry in the case of $N=1$ can be interpreted as the geometry arising from a projective structure.

We will also have need of the notion of conformally related connections
\begin{equation}
\label{eq4.4}
{\tilde \nabla}_X Y =\nabla_X Y-\frac{1}{(n-2)}\left [ X(f)Y+Y(f)X- g(X,Y)\nabla f \right ] .
\end{equation}
If $\nabla$ is the Levi-Civita connection of $g$ then  ${\tilde \nabla }$ is the Levi-Civita connection of
\begin{equation}
\label{eq4.5}
{\tilde g}:=e^{-\frac{2f}{(n-2)}}g\ .
\end{equation}
The curvature of ${\tilde g}$ is given by
\begin{equation}
\label{eq4.6}
\begin{split}
{\tilde R}(X,Y)Z=&\, R(X,Y)Z +\frac{1}{(n-2)} \left [ \left ( \hess f(Y,Z) +\frac{1}{(n-2)}\nabla_Y f \nabla_Z f\right )X\right .\\
&\, \qquad \left .  - \left (\hess f(X,Z) - \frac{1}{(n-2)}\nabla_X f \nabla_Z f\right )Y \right ]\\
&\, -\frac{1}{(n-2)} \left [ g(X,Z) \left ( g^{-1}(\cdot, \hess f(\cdot,Y)) +\frac{1}{(n-2)} \nabla f \nabla_Y f \right )\right . \\
&\, \qquad \left . - g(Y,Z) \left ( g^{-1}(\cdot,\hess f(\cdot,X)) +\frac{1}{(n-2)} \nabla f \nabla_X f \right )\right ] \\
&\, -\frac{1}{(n-2)^2} \left [ g(Y,Z)X - g(X,Z)Y \right ] \vert df \vert_g^2\ .
\end{split}
\end{equation}

We call a parametrized curve an $f$-geodesic if it is a geodesic for the connection $\nabla^f$. We will refer to the usual geodesics for the Levi-Civita connection as $g$-geodesics, while geodesics of the connection ${\tilde \nabla}$ will be called ${\tilde g}$-geodesics.
\begin{lemma}\label{lemma4.2}
Let $\gamma:[a, b) \rightarrow M$ and define
\begin{equation}
\label{eq4.7}
s(t) = \int_a^t e^{\frac{-2f \circ \gamma(r)}{\alpha}} dr\  ,\ t\in (a,b)\
\end{equation}
for a constant $\alpha$.
\begin{enumerate}
\item If $\gamma$ is a $g$-geodesic, $\alpha= n-1$ and $\gamma = \sigma \circ s$ then $\sigma$ is an $f$-geodesic.
\item If $\gamma$ is a null $g$-geodesic, $\alpha=n-2$ and $\gamma = \sigma \circ s$ then $\sigma$ is a null ${\tilde g}$-geodesic.
\end{enumerate}
\end{lemma}

\begin{proof} Let ${\hat \nabla}$ denote $\nabla^f$ or ${\tilde \nabla}$ as appropriate. In either case, direct computation using ${\hat \nabla}_X X=0$ and either \eqref{eq4.2} or \eqref{eq4.4} with $g(X,X)=0$ yields
\begin{equation}
\label{eq4.8}
{\hat \nabla}_X X =-\frac{2}{\alpha} X(f) X
\end{equation}
where $\alpha=n-1$ if ${\hat \nabla}=\nabla^f$ and $\alpha=n-2$ if ${\hat \nabla}={\tilde \nabla}$. This can be written as
\begin{equation}
\label{eq4.9}
\begin{split}
&\, {\hat \nabla}_{\hat X}{\hat X}= 0\ , \\
&\, X=: e^{-\frac{2f(t)}{\alpha}}{\hat X}\ .
\end{split}
\end{equation}
Finally, let $X=\frac{d\gamma}{dt}$, ${\hat X}=\frac{d\sigma}{ds}$, where $\gamma(t) = (\sigma\circ s) (t)$. Thus $\frac{ds}{dt}=e^{-\frac{2f(t)}{\alpha}}$. Integrate.
\end{proof}

We now turn to a brief discussion of Jacobi fields along geodesics of the weighted connection and null geodesics of the conformal connection.

\begin{lemma}\label{lemma4.3}
If $A$ is a Jacobi tensor field defined by the connection $\nabla$ along the timelike or null $g$-geodesic $\gamma(t)$ then
\begin{equation}
\label{eq4.10}
A\mapsto {\hat A} = e^{-f/\alpha}A\ .
\end{equation}
is a Jacobi tensor field with respect to the connection ${\hat \nabla}$ along the reparametrized geodesic $\sigma$ where $\gamma=\sigma\circ s$. Here either
\begin{enumerate}
\item ${\hat \nabla}=\nabla^f$, $\alpha=n-1$, and $\gamma$ is a timelike $g$-geodesic, or
\item ${\hat \nabla}={\tilde \nabla}$, $\alpha=n-2$, and $\gamma$ is a null $g$-geodesic.
\end{enumerate}
\end{lemma}

\begin{proof}
Recall that Jacobi tensors are $(1,1)$-tensor fields along $\gamma$ that are orthogonal to $\gamma'$ and obey $A''(t)+{\bar R}A(t)=0$, where $R(A):=R(A,\gamma')\gamma'$ and the overhead bar indicates that we take the quotient by $\gamma'$ (which is a necessary additional step when $\gamma'$ may be null).

Using $X=\gamma'(t)$ and ${\hat X}$ as given by \eqref{eq4.9}, we define either that ${\hat R}(A)= R^{\nabla^f}(A,{\hat X} ){\hat X}$ and we use \eqref{eq4.2} to compute it, or we define ${\hat R}(A)= R^{\nabla^f}(A,{\hat X}){\hat X}$ and we use \eqref{eq4.6} (with ${\tilde X}$ null in this case). Either way, a short calculation results in
\begin{equation}
\label{eq4.11}
{\bar {\hat R}}({\hat A}) = e^{3f/\alpha} \left \{ {\bar R}(A) +e^{-f/\alpha} \left ( \nabla_X\nabla_X e^{f/\alpha}\right ) A \right \}\ ,
\end{equation}
where ${\bar \id}$ denotes the identity on the quotient space.

On the other hand, a simple calculation using \eqref{eq4.10} and the reparametrization \eqref{eq4.7} yields
\begin{equation}
\label{eq4.12}
\frac{d^2{\hat A}}{ds^2} = e^{3f/\alpha}\left \{ A''(t)-\frac{1}{\alpha} \left [ f''(t) +\frac{(f'(t))^2}{\alpha}\right ] A(t)\right \} \,
\end{equation}
where as usual $f(t):=(f\circ\gamma)(t)$. Combining \eqref{eq4.11} and \eqref{eq4.12}, we obtain
\begin{equation}
\label{eq4.13}
\frac{d^2{\hat A}}{ds^2}+{\bar {\hat R}}({\hat A})=e^{3f/\alpha}\left \{ A''(t)+{\bar R} (A(t))\right \} =0 \ ,
\end{equation}
verifying that ${\hat A}$ as defined in \eqref{eq4.10} obeys the equation of a Jacobi tensor with respect to the connection ${\hat \nabla}$.
\end{proof}

Using (\emph{cf} equation \eqref{eq2.8})
\begin{equation}
\label{eq4.14}
{\bar {\hat B}}:= \frac{d{\hat A}}{ds}{\hat A}^{-1}-\nabla_{\frac{d\sigma}{ds}}f\ .
\end{equation}
and \eqref{eq4.10} we now immediately obtain the following result.

\begin{lemma}\label{lemma4.4}
\begin{equation}
\label{eq4.15}
\begin{split}
{\bar {\tilde B}}=&\, e^{2f/\alpha}{\bar B}_f\ , \\
{\bar {\tilde \sigma}}=&\, {\bar \sigma}\equiv{\bar \sigma}_f\ , \\
{\bar {\tilde \theta}}=&\, e^{2f/\alpha}{\bar \theta}_f \equiv  e^{2f/\alpha}\left ( {\bar \theta}-\nabla_{\gamma'}f\right )\ .
\end{split}
\end{equation}
\end{lemma}
A hypersurface $S$ is \emph{totally umbilic} if $B=Fh$ for a function $F:S\to {\mathbb R}$, where $h$ is the induced metric on $S$ ($h$ is degenerate if $S$ is null). If a hypersurface $S$ obeys ${\bar {\hat B}}=0$ at each point, then $S$ is totally umbilic in $(M,g)$. The $t=const$ slices in the $N=1$ warped product splitting obey ${\bar {\hat B}}=0$ and are totally umbilic in $(M,g)$ (see the paragraph containing \eqref{eq3.14}). An application with ${\hat \nabla}={\tilde \nabla}$ arises in Section 5.

\subsection{Completion of Proof of Theorem \ref{theorem1.5}} Consider the twisted product metric \eqref{eq3.14} (this is greater generality than necessary; for Theorem \ref{theorem1.5} it is sufficient to begin from the warped product \eqref{eq3.18}.) There is a relation between geodesics of $(M,g)$ (of any signature) and a special class of curves in the $t=0$ hypersurface $(\Sigma, {\hat h})$. Specifically, if $\nabla$ is the Levi-Civita connection compatible with $g$ and ${\hat D}$ is the connection compatible with ${\hat h}$, and if $\eta (\lambda)=\left ( \omega(\lambda), ({\hat \sigma}\circ s)(\lambda)\right )$ is a geodesic of $(M,g)$, then a straightforward calculation using $s=\int_0^{\lambda} e^{-2f(t(\tau))/(n-1)}d\tau$ shows that
\begin{equation}
\label{eq4.16}
\nabla_{\eta'(\lambda)}\eta'(\lambda) = 0 \Longleftrightarrow \begin{cases} {\hat D}_{\sigma'(s)}\sigma'(s)={\hat h} (\sigma'(s),\sigma'(s))\frac{{\hat D}f}{(n-1)}\ ,\\
\omega''(\lambda)=-\frac12 {\hat h} (\sigma'(s),\sigma'(s)) \frac{\partial}{\partial t} e^{-2f/(n-1)}\ , \end{cases}
\end{equation}
where $s=s(\lambda)$ and $f(t,x)=f(t(\lambda),x(\lambda))$. Furthermore, it follows from the equation on the top right of \eqref{eq4.16} that
\begin{equation}
\label{eq4.17}
{\hat h}(\sigma'(s),\sigma'(s))=Ae^{2f/(n-1)}\ ,
\end{equation}
where $A$ is independent of $s$ and otherwise arbitrary. We may take $A=1$, and then the equations on the right of \eqref{eq4.16} reduce to
\begin{equation}
\label{eq4.18}
\begin{split}
{\hat D}_{\sigma'(s)}\sigma'(s)=&\, -\frac12 {\hat D} e^{2f/(n-1)}\\
\omega''(\lambda)=&\, \frac{1}{(n-1)}\frac{\partial f}{\partial t}\ .
\end{split}
\end{equation}
One can define a map ${\widehat {\exp}}_p:T_p\Sigma\to \Sigma$ which sends a vector $v\in T_p\Sigma$ to the point in $\Sigma$ at parameter distance $s=|v|_{\hat h}$ along the solution curve $\sigma$ of this differential system, where $\sigma$ has initial tangent vector $v=\sigma'(0)$ at $p=\sigma(0)$. The resulting curve is the projection of an $(M,g)$-geodesic in $\Sigma$. Conversely, for a given $u=(w_0,v_0)\in T_pM$ one can first find $\sigma'(s)$ by solving the top equation in \eqref{eq4.18} subject to $\sigma'(0)=v_0$ and then, denoting $W(t):=\omega'(t)$, one can solve
\begin{equation}
\label{eq4.19}
W'(\lambda)=-\frac12 \frac{\partial}{\partial t} e^{2f/(n-1)}
\end{equation}
subject to $W'(0)=W_0$ to obtain $(W(\lambda),v(\lambda))$. Integrating $\eta'(\lambda)=(W(\lambda),v(\lambda))$ with $\eta(0)=p$ then yields a unique geodesic lift for $\sigma$ in $(M,g)$. The geodesic will be timelike, spacelike, or null depending on whether $(W_0,v_0)$ is timelike, spacelike, or null.

We now specialize to the local warped product splitting \eqref{eq3.18}. Then we can replace ${\hat h}$ in the above paragraph by $h$ and take $f=f(t)$; i.e., $Df=0$. Then the right-hand side of the top equation of \eqref{eq4.18} vanishes and $\sigma(s)$ is an $h$-geodesic. Likewise, the map $\widehat{\exp}_p$ becomes just the usual exponential map defined by unit speed $h$-geodesics parametrized by $s$.

To finish the proof of Theorem \ref{theorem1.5}, it is necessary to modify the local-to-global splitting argument of \cite[pp 558--561]{BEE}. There are two main ingredients in the argument: (i) techniques to extend tubular neighborhoods about geodesics and (ii) parallel transport as a means of ensuring Busemann functions extend along the extended geodesics and join up properly to Busemann functions defined on neighboring tubes. In short, we accomplish the former by finding $h$-geodesics in $\Sigma$. These can be lifted to timelike $g$-geodesics. Our timelike completeness and $f$-completeness assumptions then ensure that the original $h$-geodesics can be extended. To accomplish the latter, we use parallel transport with respect to the $\nabla^f$ connection. By a path independence property described in \cite[p 557]{BEE}, it does not matter than the paths chosen for the transport are not usually $\nabla^f$-geodesics and are sometimes $h$- or $g$-geodesics.

In slightly greater detail, as in \cite[p 558, first paragraph]{BEE} let $p_0$ lie on the timelike line $\gamma_0$ and let $U_0\simeq ({\mathbb R}\times \Sigma, -dt^2\oplus f^2 h)$ be a tubular neighborhood about $\gamma_0$. Letting $\Sigma_0$ denote the $t=0$ embedded image of $\Sigma$, if $\edge(\Sigma_0)$ is non-empty, choose a sequence of points $q_n\in \Sigma_0$ approaching $\edge(\Sigma_0)$ and find $h$-geodesics $\exp_p(sv_n)\in \Sigma_0$ joining $p_0$ to each $q_n$, where $s\in [0,a_n]$ and $h(v_n,v_n)=1$. Find the limiting initial unit tangent vector $v=\lim_n v_n$ and construct the geodesic $\sigma:[0,a)\to \Sigma_0 : s\mapsto \exp_{p_0}(s,v)$. Lift this, using the above procedure, to a timelike geodesic $\eta:[0,b)\to M$ (where $a=s(b)$). By timelike geodesic completeness, $\eta$ can be extended to $\eta(b)$, so $\sigma$ extends to $\sigma(a)\in \edge(\Sigma_0)$.

A simple calculation on the tubular neighborhood $U_0$ yields
\begin{equation}
\label{4.20}
\begin{split}
&\, \nabla \frac{\partial}{\partial t} =\frac{1}{(n-1)}f'(t)\left ( \id - dt\otimes{\partial}{\partial t}\right )\\
\implies &\, \nabla^f \frac{\partial}{\partial t} =-\frac{2}{(n-1)}f'(t) dt\otimes{\partial}{\partial t}\\
\implies &\, \nabla^f \left ( e^{2f(t)/(n-1)}\frac{\partial}{\partial t}\right )= 0\ .
\end{split}
\end{equation}
Thus $P:= e^{2f(t)/(n-1)}\frac{\partial}{\partial t}$ is $\nabla^f$-parallel in $U$. In particular, $P$ is the unique vector field obtained by $\nabla^f$-parallel-transporting along $\sigma:[0,a)\to \Sigma_0$ the vector $e^{2f(0)/(n-1)}\frac{\partial}{\partial t}\big\vert_{p_0}$ based at $p_0$. Since $\sigma$ extends to $\sigma(b)\in \edge(\Sigma_0)$, so does $P$. Although $\sigma$ is not geodesic with respect to $\nabla^f$, this does not matter since $P$ is globally $\nabla^f$-parallel in $U$, for by a simple argument (see \cite[(14.44) p 557]{BEE}), the extension of $P$ to $\edge(\Sigma_0)$ is indeed path-independent and thus well-defined. Next define $N\big\vert_{q_n}:=e^{-2f(0)/(n-1)}P(q_n)\equiv \frac{\partial}{\partial t}\big\vert_{q_n}$ (recall $\Sigma_0\ni q_n \to p_1$) and define $N\big\vert_{p_1}:=e^{-2f(0)/(n-1)}P(p_1)$. As in \cite{BEE}, at each $q_n$ we can use the exponential map for $g$-geodesics to obtain timelike lines $\exp_{q_n}(tN\vert_{q_n})$ orthogonal to $\Sigma_0$, and then $\gamma_{p_1}(t):=\exp_{p_1}(tN\vert_{p_1})$ will also be a timelike line orthogonal to $\Sigma_0$. Having proved local splitting about a timelike line in Section 3, we can apply this result now to obtain a local splitting in a tube $U_1$ about $\gamma_{p_1}$.

We now paraphrase the next step in the argument in \cite{BEE} as follows. One can now define two fields $P$ as above, namely, the original field, say $P_0\equiv P$, constructed by $\nabla^f$-parallel transport of the vector $e^{2f(0)/(n-1)}\frac{\partial}{\partial t}\big\vert_{p_0}$ based at $p_0$ and the new field $P_1$ constructed by $\nabla^f$-parallel transport of the vector $e^{2f(0)/(n-1)}\frac{\partial}{\partial t}\big\vert_{p_1}$ based at $p_1$. But since the two base vectors here are also related by $\nabla^f$-parallel transport, $P_1$ is derived from the same transport process as $P_0$, both beginning with the same base vector at $p_0$, except that the path that gives $P_1$ must pass through $p_1$. By the path independence property, the resulting vector fields agree everywhere on $U_0\cap U_1$, and so do the related Busemann functions.

Indeed, the entire remainder of the argument in \cite{BEE} extending the local splitting to a global one follows by replacing parallel transport with the Levi-Civita connection $\nabla$ by parallel transport with $\nabla^f$ at each step in \cite{BEE}.

Finally, it is clear that $h$ is a complete metric on the spacelike factor $\Sigma$ for, if it were not, then there would be an inextendible $h$-geodesic of finite arclength. Let this geodesic $\sigma(s)$ have initial endpoint $p=\sigma(0)$. Then it lifts to a timelike geodesic with initial tangent $v(0)=(2,\sigma'(0))$ at $p$. Since the proper time $\lambda$ along this geodesic is related to the arclength $s$ of $\sigma$ by $s=\int_0^{\lambda} e^{-2f(t(\tau))/(n-1)}d\tau$, the condition that $\sigma$ extends to arbitrarily large $s$ is precisely the $f$-completeness criterion for its lift. Hence incompleteness of $\sigma$ would imply a violation of timelike $f$-completeness, a contradiction. Thus, $h$ is a complete metric on $\Sigma$ and the proof of the $N=1$ case of Theorem \ref{theorem1.5} is now finished.

\section{The null splitting theorem}
\setcounter{equation}{0}

\noindent We recall Galloway's null splitting theorem:

\begin{theorem}[Galloway, \cite{Galloway}]\label{theorem5.1}
Let $(M,g)$ be a null geodesically complete spacetime which obeys $\ric(X,X)\ge 0$ for all null vectors $X$ and contains a null line $\eta$. Then $\eta$ is contained in a smooth, closed, achronal, totally geodesic null hypersurface.
\end{theorem}

We note that under a conformal transformation of the form \eqref{eq4.5}, the Ricci tensor transforms as
\begin{equation}
\label{eq5.1}
\begin{split}
\ric_{\tilde g} = &\, \ric_g + \hess_g f +\frac{1}{(n-2)}df\otimes df +\frac{1}{(n-2)}\left [ \Delta_g f +\left \vert df \right \vert_g^2 \right ] g\\
= &\, \ric_f^2 +\frac{1}{(n-2)} \left ( \Delta_{-f} f\right ) g\ ,
\end{split}
\end{equation}
where $\Delta_f :=\Delta - \nabla_{\nabla f}$ is the \emph{drift Laplacian}.
Furthermore, a simple calculation shows that
\begin{equation}
\label{eq5.2}
\ric_f^2 = \ric_f^N + \frac{(2-N)}{(n-N)(n-2)}df\otimes df\ ,
\end{equation}
so we can write
\begin{equation}
\label{eq5.3}
\ric_{\tilde g} = \ric_f^N + \frac{(2-N)}{(n-N)(n-2)}df\otimes df+\frac{1}{(n-2)} \left ( \Delta_{-f} f\right ) g\ .
\end{equation}

\begin{lemma}\label{lemma5.2} If $N\in [-\infty,2]\cup (n,\infty]$, $\ncd(N)\implies \ric_{\tilde g}(X,X)\ge 0$ for all null $X$.
\end{lemma}

\begin{proof}
Immediate from \eqref{eq5.3}.
\end{proof}

We are now ready to prove our null splitting theorem.

\begin{proof}[Proof of Theorem \ref{theorem1.6}] We are given that $(M,g)$ admits a null line $\eta$; i.e., an inextendible, achronal geodesic. It remains achronal after a conformal transformation \eqref{eq4.5} and remains geodesic ${\tilde \eta}$ after reparametrization, where $\eta={\tilde \eta}\circ s$ with $s$ given by \eqref{eq4.7}. If $(M,g)$ is null geodesically complete and $f$-complete, then $(M,{\tilde g})$ is null geodesically complete. Finally, since $(M,g)$ obeys $\ncd(N)$ for $N\in [-\infty,2]\cup (n,\infty]$, Lemma \ref{lemma5.2} implies that $(M,{\tilde g})$ obeys $\ric_{\tilde g}(X,X)\ge 0$ for all null $X$. Hence $(M,{\tilde g})$ fulfils the conditions of Theorem \ref{theorem5.1}.

Hence ${\tilde \eta}$ is contained in a smooth, closed, achronal, totally ${\tilde g}$-geodesic null hypersurface ${\tilde B}=0$. Then
\begin{equation}
\label{eq5.7}
0={\tilde B}=e^{\frac{2f}{(n-2)}}B_f \ \implies\ B_f =0\ \implies\ \sigma_f\equiv\sigma=0\ \text{and}\ \theta=\nabla_{\gamma'} f
\end{equation}
along any null geodesic generator $\gamma$ of $S$, where the last implication uses Lemma \ref{lemma4.4}. This proves Theorem \ref{theorem1.6}.(ii). We further note that from the Raychaudhuri equation \eqref{eq2.9} and $\ncd(N)$, we see that $\ric_f^N(\gamma',\gamma')=0$, while from \eqref{eq2.10} we see that the $f$-generic condition fails.

When $N\in [-\infty,2)\cup (n,\infty]$ (i.e., $N\neq 2$) equation \eqref{eq2.9} and $\ncd(N)$ also imply that $\nabla_{\gamma'}f=0$. This proves Theorem \ref{theorem1.6}.(i). Further we then obtain that $\ric(\gamma',\gamma')=0$ along the null generators of $S$, and as well ${\bar R}=0$, so the the generic condition fails along $\gamma$. \end{proof}

\begin{remark}
In the case $N=2$, it is not possible to obtain any rigidity of the function $f$.  To see this simply let $(M, \tilde g)$ be any null geodesically complete spacetime which obeys $\ric(X,X)\ge 0$ for all null vectors $X$ and contains a null line. Let $f$ be any smooth bounded function on $M$ and let $g=e^{\frac{2f}{n-2}}$.  Then $(M,g)$ will satisfy all of the hypotheses of Theorem \ref{theorem1.6}
\end{remark}

\end{document}